\PassOptionsToPackage{pdfpagelabels=false}{hyperref}
\documentclass[twoside]{aiml18}

\usepackage{aiml18macro}

\usepackage{graphicx}
\usepackage{amsmath}
\usepackage{amssymb}
\usepackage{mathtools}
\usepackage{tikz, float}
\usepackage[british]{babel}
\usepackage{cleveref}
\usepackage{enumerate}
\usepackage{marginnote}
\usepackage{algorithm}
\usepackage{algpseudocode}
\usetikzlibrary{calc}
\usetikzlibrary{decorations.pathmorphing}

\DeclareMathOperator{\down}{\downarrow}
\DeclareMathOperator{\up}{\uparrow}
\DeclareMathOperator{\Cl}{Cl}

\DeclareMathOperator{\join}{\oplus}

\newcommand \limplies {\mathbin{\rightarrow}}

\newcommand \aand {\mathbin{\wedge}}

\newcommand{\from}{\colon}

\newcommand{\modal}{\mathbf}
\newcommand{\F}{\modal F}
\renewcommand{\P}{\modal P}

\newcommand{\defnn}[1]{\textbf{#1}}

\newcommand\restr[1]{{\restriction_{#1}}}
\newcommand\set[1]{\{#1\}}

\newcommand\mcs{\mathsf{MCS}}

\renewcommand{\vec}[1]{\boldsymbol{#1}}

\renewcommand{\algorithmicindent}{2.65mm}

\DeclarePairedDelimiter\abs{\lvert}{\rvert}

\def\smallskip{\vskip\smallskipamount}
\def\medskip{\vskip\medskipamount}
\def\bigskip{\vskip\bigskipamount}
\def\restate{\par\medskip\noindent\hbox{\bf Theorem~\ref{pspace}}\it\ }  
\def\endrestate{
\par\medskip}
\newenvironment{restate_pspace}{\restate}{\endrestate}

\def \tikzconst{2.4}

\def\reals {\mathbb{R}}

\shttltrue
\def\shtitle{The temporal logic of two-dimensional spacetime with slower-than-light accessibility}

\def\lastname{Hirsch and McLean}

\begin{document}

\begin{frontmatter}
  \title{The temporal logic of two-dimensional Minkowski spacetime with slower-than-light accessibility is decidable}
  \author{Robin Hirsch}\footnote{r.hirsch@ucl.ac.uk}
  \author{Brett McLean}\footnote{b.mclean@cs.ucl.ac.uk}
  \address{Department of Computer Science \\ University College London\\Gower Street, London WC1E 6BT}
  
  \begin{abstract}
  We work primarily with the Kripke frame consisting of two-dimen\-si\-on\-al Minkowski spacetime with the irreflexive accessibility relation `can reach with a slower-than-light signal'. We show that in the basic temporal language, the set of validities over this frame is decidable. We then refine this to PSPACE-complete. In both cases the same result for the corresponding reflexive frame follows immediately. With a little more work we obtain PSPACE-completeness for the validities of the Halpern--Shoham logic of intervals on the real line with two different combinations of modalities. 
  \end{abstract}

 \end{frontmatter}




\section{Introduction}
 
Minkowski spacetime refers to the flat spacetime of special relativity, where in any inertial coordinates $(\vec r, ct)$, light travels in straight lines at a $45^\circ$ angle to the time axis.  To view this as a Kripke frame, there are at least four natural accessibility relations to chose from: reflexive and irreflexive versions of `can reach with a lightspeed-or-slower signal' and `can reach with a slower-than-light signal'. See below for the formal  definitions in the two-dimensional case.  

 For the basic \emph{modal} language, Goldblatt found that, regardless of the number of spatial dimensions, both reflexive choices produce the logic $\mathbf{S4.2}$ (reflexivity, transitivity, confluence) \cite{Goldblatt1980}. Shapirovsky and Shehtman proved the irreflexive slower-than-light logic is $\mathbf{OI.2}$---transitivity, seriality, confluence, \emph{two-density} (see \Cref{section2})---again regardless of dimension \cite{DBLP:conf/aiml/ShapirovskyS02}. Both $\mathbf{S4.2}$ and $\mathbf{OI.2}$ are $\mathsf{PSPACE}$-complete \cite{Shapirovsky2005-SHAOPI}.
In contrast, Shapirovsky has shown that with irreflexive \emph{exactly lightspeed} accessibility, validity is undecidable  \cite{shapirovsky2010simulation}.

For the basic \emph{temporal} language, the problems of axiomatising and determining the complexity of the validities of these frames had all been open for decades---Shehtman recommended an investigation of the two-dimensional case (one time and one space dimension) in the concluding remarks of \cite{Shehtman1983}.
In the two-dimensional case, starting with coordinates $(r, ct)$,  we may rotate 
 the axes through 45$^\circ$ to get  coordinates $(x, y)$, where $x=\frac{1}{\sqrt2}(ct+r),\; y=\frac{1}{\sqrt2}(ct-r)$.  With these coordinates the reflexive   lightspeed-or-slower relation $\leq$, and  the irreflexive  slower-than-lightspeed relation $\prec$, are given by 
\begin{align}\label{order}
\!\!\!\!\!\!(x, y)\leq (x', y')&\!\iff\! x\leq x'\wedge y\leq y',&
(x, y)\prec (x', y')&\!\iff\! x<x'\wedge y<y
\end{align}
where $<, \leq$ on the right are the usual irreflexive/reflexive orderings of the reals, and $<, \preceq$ are obtained from $\leq, \prec$ by deleting/adding the identity, respectively.

 Recently, Hirsch and Reynolds managed to show that for this two-dimensional case, with either reflexive or irreflexive lightspeed-or-slower accessibility, the validity problem is $\mathsf{PSPACE}$-complete \cite{hrminkowski}. However, they were unable to obtain decidability/complexity results for slower-than-light accessibility. Indeed decidability appears in item (2) in the list of open problems at the end of their paper. In this paper we solve Hirsch and Reynolds' problem, eventually proving the following.

\begin{restate_pspace}
 On the frame consisting of two-dimensional Minkowski spacetime equipped with the irreflexive slower-than-light accessibility relation, the set of validities of the basic temporal language is $\mathsf{PSPACE}$-complete. The same is true with reflexive slower-than-light accessibility.
\end{restate_pspace}

The proof of Theorem~\ref{pspace} follows that of \cite{hrminkowski} very closely. The only additional insight needed is that all pertinent information about the behaviour of a valuation on a light-line can be captured in a finite way---see Definition~\ref{bi-trace}---, despite all of a light-line's points being mutually inaccessible.   


The proof of the main result is structured as follows.
\begin{enumerate}
\item
Given a fixed formula $\phi$ whose satisfiability is to be determined, we define the maximal consistent sets of subformulas/negated subformulas of $\phi$, where consistency is with respect to the class of all temporal frames.

\item
We define \emph{surrectangles} by recording a maximal consistent set at each point of a rectangle, together with some information about the maximal consistent sets holding near but beyond the boundaries of the rectangle.
\item
Starting in \Cref{section_biboundary}, we define \emph{biboundaries}, also based on the maximal consistent sets. Biboundaries have finite specifications and the intuition for them is as a record of the information contained near the boundary of a surrectangle, plus a little from its interior. Indeed we define the biboundary $\partial^s$ determined by a given surrectangle $s$.

\item
We define three operations on biboundaries: \emph{joins}, \emph{limits}, and \emph{shuffles}, and we define the set of \emph{fabricated biboundaries} to be those biboundaries formed by iterating these operations, starting from certain basic biboundaries. This is all computable, and the iterative procedure must terminate, for the biboundaries are finite in number.

\item
We show that every fabricated biboundary is given by $\partial^s$ for some surrectangle $s$, by describing a recursive construction of $s$. (\Cref{from_biboundary})

\item
Conversely, we show that for every surrectangle $s$, the biboundary $\partial^s$ is fabricated. (\Cref{from_surrectangle})

\item
We deduce the decidability of the validity problem over our frame.

\item
In \Cref{section_pspace}, we give a procedure for deciding if a biboundary is fabricated using polynomial space, and also note the satisfiability task is $\mathsf{PSPACE}$-hard. Consequently our result that validity is decidable is refined to validity being $\mathsf{PSPACE}$-complete.
\end{enumerate}

In \cite{Halpern:1991:PML:115234.115351}, Halpern and Shoham introduced a family of modal logics in which the entities under discussion are intervals. These logics have subsequently come to be highly influential, and their axiomatisability, decidability, and complexity extensively studied \cite{venema1990,10.1007/978-3-540-89439-1_41,5970233,BMSS11,6311113,MM13,Bresolin:2017:HFH:3130378.3105909}. 
 In \Cref{section_intervals}, we describe how the proof of the complexity of validity for two-dimensional Minkowski spacetime can be adapted to prove $\mathsf{PSPACE}$-completeness of the temporal logic of intervals where the accessibility relation is $\mathtt{overlaps} \cup \mathtt{meets} \cup \mathtt{before}$, or its reflexive closure.

\section{Preliminaries and surrectangles}\label{section2}

We often, but not exclusively, follow the terminology and notation of \cite{hrminkowski}.
We take as primitive the propositional connectives $\neg$ and $\lor$, and modal operators $\modal F$ and $\modal P$; the usual abbreviations apply. The semantics is the usual semantics on temporal frames---Kripke frames for which the accessibility relation for $\modal P$ is the converse of that for $\modal F$.  
An example of a formula that is valid over  two-dimensional slower-than-light frames but not over irreflexive  lightspeed-or-slower frames is the two-density formula \begin{align*}(\F p\wedge\F q)\rightarrow\F(\F p\wedge\F q)\end{align*}
asserting that if $x$ precedes both $z_1$ and $z_2$, then there exists $y$ that is between $x$ and $z_1$ \emph{and} between $x$ and $z_2$.

Throughout, $\phi$ is a fixed formula whose satisfiability is to be determined. The \defnn{closure} $\Cl(\phi)$ of $\phi$ is the set of all subformulas and negated subformulas of $\phi$.  A \defnn{maximal consistent set} is a subset of $\Cl(\phi)$ that is satisfiable in some temporal frame and is maximal with respect to the inclusion ordering, subject to the satisfiability constraint. We denote the set of maximal consistent sets by $\mcs$. It is well known that satisfiability in a temporal frame is decidable,
 indeed $\mathsf{PSPACE}$-complete \cite{Spaan1993}. Hence the set of all maximal consistent sets of $\phi$ is a (total) computable function of $\phi$. The relation on $\mcs$ given by
\begin{align*}
m\lesssim n &\iff &&\forall\F\psi\in \Cl(\phi)\;  ((\psi\in n\limplies \F\psi\in m)\aand(\F\psi\in n\limplies \F\psi\in m))    \\  
&&\aand\,&
\forall\P\psi\in \Cl(\phi)\;   ((\psi\in m\limplies \P\psi\in n)\aand(\P\psi\in m\limplies \P\psi\in n))
\end{align*}
is transitive, and so defines a preorder on reflexive elements. We call a $\lesssim$-equivalence class a \defnn{cluster} and write $\leq$ for the partial order on clusters induced by $\lesssim$. The notation $<$ means $\leq$ but not equal. We extend these notations to compare a cluster $c$ and a (not necessarily $\lesssim$-reflexive) maximal consistent set $m$ as follows. Write $m \leq c$ if for all $n \in c$ we have $m \lesssim n$, and write $m < c$ if for all $n \in c$ we have ($m \lesssim n \aand m \neq n)$. Similarly for $c \leq m$ and $c < m$.

\begin{definition}
A formula of the form $\F \psi$ is a \defnn{future defect} of a maximal consistent set $m$ if $\F \psi \in m$. A future defect of a set $S$ of maximal consistent sets is a future defect of any member of $S$ \emph{unless} we explicitly identify $S$ as a cluster. The formula $\F \psi$ is a future defect of a cluster $c$ if $\F \psi$ is contained in some $m \in c$, but, for all $n \in c$, we have $\psi \not\in n$. A future defect $\F \psi$ is \defnn{passed up} to a set $S$ of maximal consistent sets if either $\psi$ or $\F \psi$ belongs to some $m \in S$. A \defnn{past defect} is defined similarly.
\end{definition}

Recall from \eqref{order} the ordering $\prec$ and the associated  `can reach with a slower-than-light signal'  Kripke frame  $(\reals^2, \prec)$. Subsets of $\mathbb{R}^2$ inherit the same ordering. The notation $\up\vec x$ denotes the set $\{\vec y \mid \vec x \prec \vec y\}$. (The set $\vec y$ is drawn from should be clear.) We define the operations $\vee$ and $\wedge$ by \begin{align*}
(x, y)\vee (x', y') &= (\operatorname{max}\{x, x'\},
\operatorname{max}\{y, y'\}),\\ (x, y)\wedge (x', y') &= (\operatorname{min}\{x, x'\}, \operatorname{min}\{y, y'\}).\end{align*} (According to lightspeed-or-slower accessibility $\vee$ is join and $\wedge$ is meet). 
 We also define the partial order $\triangleleft$ on $\mathbb{R}^2$ by $(x, y)\triangleleft (x', y')\iff x\leq x'$, $y\geq y'$, and $(x, y)\neq (x', y')$. When this holds we say that $(x, y)$ is `northwest' of $(x', y')$.

The \defnn{diagonal dual} of a condition $c$ on a frame/model whose domain is a subset of $\mathbb{R}^2$ is the condition that $c$ holds on the frame/model obtained by swapping the $x$- and $y$-axes (which will not affect accessibility). 
 The \defnn{temporal dual} is the result of reversing the $x$-axis and reversing the $y$-axis (which reverses accessibility), and also swapping $\F$ and $\P$ in formulas. 
When we say `all duals' we mean the diagonal dual, the temporal dual, and the diagonal temporal dual.

Suppose we have a preorder-preserving map $f$ from a subset of $\mathbb R^2$ to $\mcs$  
 and that $f$ is defined on $\mathcal U \cap \up \vec x$ for some $\mathcal U$ an open neighbourhood (in $\mathbb R^2$) of $\vec x$. Then it is easy to see there is a unique cluster $c$ such that there exists a neighbourhood $\mathcal U'$ of $\vec x$ such that $f$ only takes values in $c$ on $\mathcal U' \cap \up \vec x$. We denote this cluster by $f^+(\vec x)$. The analogous value for $\down \vec x$ we denote $f^-(\vec x)$.

Let $(x, y), (x', y)$ be two distinct points on a horizontal line segment in the domain of the preorder-preserving map $f$.  Then $(x, y)\not\prec(x', y)$ so we do not know  the relation between $f(x, y)$ and $f(x', y)$. However, it is easy to see that $x\leq x' \implies f^+(x, y)\lesssim f^+(x', y)$.  Similarly for $f^-$. If $f^+$ is constantly equal to $c$ on an open line segment  $l$ and defined at the   left end of $l$ then $f^+$ also equals $c$ there.\footnote{For line segments, `open' and `closed' have their usual meanings of `excludes end points' and `includes end points', respectively.} Similarly for $f^-$ with right ends. All diagonal duals of statements in this paragraph (that is, statements for vertical lines) hold similarly.

The following concept will be used in definitions that follow, and informs the way we think of a bi-trace (Definition~\ref{bi-trace}) as specifying behaviour \emph{on} a horizontal or vertical line segment.
\begin{definition}
Let $c^-$ and $c^+$ be clusters. The \defnn{interpolant} of $c^-$ and $c^+$ is the set of all $m \in \mcs$ with $c^- \leq m \leq c^+$ such that all future defects of $m$ are passed up to $c^+$, and all past defects of $m$ are passed down to $c^-$.
\end{definition}

\begin{definition}\label{bi-trace}
A \defnn{bi-trace} (of length $n$) is two sequences $c_0^+ \leq \dots \leq c_n^+$ and $c_0^- \leq \dots \leq c_n^-$ of clusters, and one sequence $b_1, \dots, b_n \in \mcs$, with the following constraints.
\begin{itemize}
\item
For each $i \leq n$ we have $c_i^- \leq c_i^+$, and the interpolant of $c_i^-$ and $c_i^+$ is nonempty.

\item
For each $i < n$ either $c_i^- < c_{i+1}^-$ or $c_i^+ < c_{i+1}^+$.
\item
For each $i < n$ we have $c_i^- \leq b_{i+1} \leq c_{i+1}^+$.

\item
For each $i < n$ all future defects of $b_{i+1}$ are passed up to $c_{i+1}^+$ and all past defects passed down to $c_i^-$.
\end{itemize}
Formally, we consider a bi-trace to be the interleaving of its three sequences, with the advantage that we can use the notation $(c_0^-, c_0^+, b_1, c_1^-, \dots, c_n^+)$ to indicate a bi-trace.
We call the $c_i^+$'s the \defnn{upper clusters}, the $c_i^-$'s the \defnn{lower clusters}, $c_0^+$ and $c_0^-$ the \defnn{initial clusters}, and $c_n^+$ and $c_n^-$ the \defnn{final clusters}. Each pair $c_i^-, c_i^+$ is a \defnn{cluster pair} of the bi-trace, and each $b_i$ a \defnn{transition value}. The top part of \Cref{fig:bi-trace} suggests how to visualise a bi-trace.  There are only finitely many bi-traces (because of the second condition in their definition).
\end{definition}

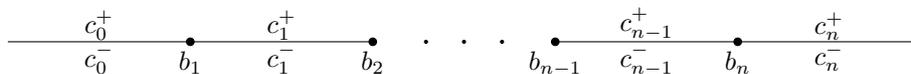
\begin{figure}[H]
\centering
\begin{tikzpicture}
\foreach \x in {0,1}
{
\draw[fill] (\tikzconst*\x+\tikzconst,0) circle [radius=0.05];
\node [above] at (\tikzconst*\x+.5*\tikzconst, -.1) {$c^+_\x$};
\node [below] at (\tikzconst*\x+.5*\tikzconst, 0.1) {$c^-_\x$};
\draw (\tikzconst*\x,0)--(\tikzconst*\x +\tikzconst,0);
}
\node [below] at (\tikzconst, 0) {$b_1$};
\node [below] at (\tikzconst*2, 0) {$b_2$};
\foreach \x in {3,4}
{
\draw[fill] (\tikzconst*\x,0) circle [radius=0.05];
\draw (\tikzconst*\x,0)--(\tikzconst*\x +\tikzconst,0);
}
\node [below] at (\tikzconst+\tikzconst*2, 0) {$b_{n-1}$};
\node [above] at (\tikzconst*3.5, -.1) {$c^+_{n-1}$};
\node [below] at (\tikzconst*3.5, 0.1) {$c^-_{n-1}$};
\node [below] at (\tikzconst+\tikzconst*3, 0) {$b_{n}$};
\node [above] at (\tikzconst*4.5, -.1) {$c^+_n$};
\node [below] at (\tikzconst*4.5, 0.1) {$c^-_n$};
\draw[fill] (2.5*\tikzconst-.5,0) circle [radius=0.02];
\draw[fill] (2.5*\tikzconst,0) circle [radius=0.02];
\draw[fill] (2.5*\tikzconst+.5,0) circle [radius=0.02];
\end{tikzpicture}
\vspace{-16pt}
\caption{A bi-trace}\label{fig:bi-trace}
\end{figure}

\begin{definition}
Let $t_1 = (c_0^-, c_0^+, b_1, \dots, c_n^+)$ and $t_2 =\allowbreak (e_0^-, e_0^+, d_1, \dots, e_k^+)$ be bi-traces, and $a \in \mcs$. Then $t_1 + a + t_2$ is defined and equal to $(c_0^-, c_0^+, b_1, \dots, c_n^+, a, e_0^-, e_0^+, d_1, \dots, e_k^+)$ if this is a bi-trace. It is also defined if the final clusters of $t_1$ equal the initial clusters of $t_2$, and $a$ is in the interpolant of $c_n^-$ and $c_n^+$, in which case it equals $(c_0^-, c_0^+, b_1, \dots, c_n^+, d_1, \dots, e_k^+)$.
\end{definition}

We now define surrectangles.  
Intuitively, the domain of a surrectangle is a rectangle plus infinitesimally more beyond any closed edges of the rectangle (a `surreal rectangle'), and a surrectangle records a valuation on this domain.  

\begin{definition}\label{surrectangle}
A \defnn{rectangle} is a product of two intervals of $\mathbb R$ (with unbounded and single-point intervals both allowed); it is \defnn{degenerate} if either interval is a single point, otherwise it is \defnn{nondegenerate}. An \defnn{edge} of a rectangle $R$ is an edge of the closure of $R$ in $\mathbb R^2$ and is not considered to include its end points; a \defnn{closed edge} of $R$ is an edge of $R$ contained in $R$.  An \defnn{upper edge} of $R$ is either a horizontal edge with maximal vertical component, or a vertical edge with maximal horizontal component, a lower edge is defined dually.
A rectangle is \defnn{open}/\defnn{closed} if it is open/closed in $\mathbb R^2$. The notation $[ \vec b,  \vec t]$, for points $\vec b = (b_1, b_2)$ and $\vec t = (t_1, t_2)$, signifies the closed rectangle $\{(x, y) \in \mathbb R^2 \mid b_1 \leq x \leq t_1\text{ and }b_2 \leq y \leq t_2\}$.

 A \defnn{surrectangle} consists of the following data.
\begin{enumerate}[(1)]
\item
A preorder-preserving map $f \from (R, \prec) \to (\mcs, \lesssim)$, for some \emph{nondegenerate} rectangle $R$. We call $R$ the \defnn{rectangle} and $f$ the \defnn{core map} of the surrectangle.

\item\label{segments}
For each closed horizontal upper edge $e$ of $R$ with endpoints $(x_0, y)$ and $(x', y)$,\footnote{Here, it could be that $x_0 = -\infty$ and/or $x' = \infty$; this does not present any problems.} a finite sequence $c_0^+ \leq \dots \leq c_n^+$ of clusters and a sequence $(x_1,  y), \ldots ,(x_{n}, y) \in e$, with $x_0 < \dots < x_n < x_{n+1}$, defining $x_{n+1}$ to be $x'$.  Similar finite sequences for any closed vertical and/or lower edges.\footnote{For lower edges, these supplementary clusters are the ones denoted $c_i^-$.}
\end{enumerate}
And is required to satisfy the following constraints.
\begin{enumerate}[(1)]\setcounter{enumi}{2}
\item \label{f-}
For each closed horizontal upper edge as above, $f^-$ is constant on each open line segment $((x_i, y), (x_{i+1}, y))$ (let this constant cluster be $c_i^-$), and $(c_0^-, c_0^+, b_1, c_1^-, \ldots, c_n^-,  c_n^+)$ forms a bi-trace, where $b_i$ is defined to be $f(x_i, y)$, for each $i$. Also, all duals of this constraint. (\Cref{fig:edge} suggests how to visualise a closed edge.)

\item\label{no_defects}
For any point $\vec x \in R$, if $\F \psi \in f(\vec x)$ either
\begin{itemize}
\item
\textbf{resolved internally:} there is $\vec y \in \up \vec x$ such that $\psi \in f(\vec y)$,

\item
\textbf{passed upwards:} $R$ has a boundary point $\vec y$ either due north or due east of $\vec x$ such that $\F \psi$ is passed up to $f(\vec y)$.
\end{itemize}
\item
The temporal dual of \eqref{no_defects} holds.
\end{enumerate}
In \eqref{segments}, we call $c_0^+ , \dots , c_n^+$ the \defnn{supplementary clusters} of $e$, we call $(x_1,  y), \ldots \allowbreak,(x_{n}, y)$ the \defnn{transition points} of  $e$, and we call $f(x_1, y), \dots, f(x_n, y)$ the \defnn{transition values}.  Note that the clusters $c_i^-$ are determined by $f^-$ in \eqref{f-}, and for edges not contained in $R$ (for example if the rectangle is unbounded in the corresponding direction) supplementary clusters and transition points are not defined.    

Let $\vec b$ and $\vec t$ be respectively the lower-left and upper-right corners of $R$. Then $f^+(\vec b)$ and $f^-(\vec t)$ are necessarily defined. The \defnn{height} of the surrectangle is the maximum possible length of a chain of clusters (not necessarily in the image of $f$) from its \defnn{lower cluster} $f^+(\vec b)$ to its \defnn{upper cluster} $f^-(\vec t)$. Surrectangles also inherit descriptions such as open/closed from their underlying rectangle.
\end{definition}

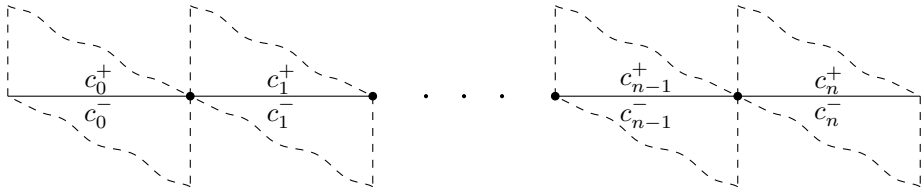
\begin{figure}[H]
\centering
\begin{tikzpicture}[decoration={snake, amplitude=1.5, segment length=30}]
\foreach \x in {0,1}
{
\draw[dashed] (\tikzconst*\x,0)--(\tikzconst*\x,1.2);
\draw[dashed, decorate] (\tikzconst*\x,1.2)--(\tikzconst*\x+\tikzconst,0);
\draw[dashed](\tikzconst*\x+\tikzconst,0)--(\tikzconst*\x+\tikzconst,-1.2);
\draw[dashed, decorate] (\tikzconst*\x+\tikzconst,-1.2)--(\tikzconst*\x,0);
\draw[fill] (\tikzconst*\x+\tikzconst,0) circle [radius=0.05];
\node [above] at (\tikzconst*\x+.5*\tikzconst, -.1) {$c^+_\x$};
\node [below] at (\tikzconst*\x+.5*\tikzconst, 0.1) {$c^-_\x$};
\draw (\tikzconst*\x,0)--(\tikzconst*\x +\tikzconst,0);
}
\foreach \x in {3,4}
{
\draw[dashed] (\tikzconst*\x,0)--(\tikzconst*\x,1.2);
\draw[dashed, decorate] (\tikzconst*\x,1.2)--(\tikzconst*\x+\tikzconst,0);
\draw[dashed](\tikzconst*\x+\tikzconst,0)--(\tikzconst*\x+\tikzconst,-1.2);
\draw[dashed, decorate] (\tikzconst*\x+\tikzconst,-1.2)--(\tikzconst*\x,0);
\draw[fill] (\tikzconst*\x,0) circle [radius=0.05];
\draw (\tikzconst*\x,0)--(\tikzconst*\x +\tikzconst,0);
}
\node [above] at (\tikzconst*3.5, -.1) {$c^+_{n-1}$};
\node [below] at (\tikzconst*3.5, 0.1) {$c^-_{n-1}$};
\node [above] at (\tikzconst*4.5, -.1) {$c^+_n$};
\node [below] at (\tikzconst*4.5, 0.1) {$c^-_n$};
\draw[fill] (2.5*\tikzconst-.5,0) circle [radius=0.02];
\draw[fill] (2.5*\tikzconst,0) circle [radius=0.02];
\draw[fill] (2.5*\tikzconst+.5,0) circle [radius=0.02];
\end{tikzpicture}
\vspace{-10pt}
\caption{A closed edge of a surrectangle}\label{fig:edge}
\end{figure}

\section{Biboundaries}\label{section_biboundary}

\begin{definition}
A \defnn{biboundary} is a partial map $\partial$ on $\{-, +, b, t, l, r,\allowbreak N, S,\allowbreak E, W\}$. It must be defined on $-$ and $+$, and be cluster-valued there. If $\partial$ is defined on $N, S, E$, or $W$, it is bi-trace-valued there, and if defined on $b, t, l$, or $r$, it is $\mcs$-valued there. It is defined on $b$ if and only if it is defined on $S$ and $W$; similarly for $t$ and $N, E$, for $l$ and $N, W$, and for $r$ and $S, E$. The following conditions must also be satisfied.
\begin{enumerate}
\item\label{coherence}
If $\partial(b)$ is defined, then $\partial(b) \leq \partial(-)$ and every future defect of $\partial(b)$ is passed up to $\partial(-)$.

\item
If $\partial(W)$ is defined, then the initial upper cluster of $\partial(W)$ equals $\partial(-)$.

\item
If $\partial(l)$ is defined, then it is less than or equal to the initial upper cluster of $\partial(N)$, and every future defect of $\partial(l)$ is passed up to that cluster.

\item\label{end}
Every future defect of $\partial(+)$ is passed up to the interpolant of the final clusters of $\partial(N)$, to the interpolant of the  final clusters of $\partial(E)$, or to $\partial(t)$, with $\partial$ defined in the places appropriate to the case.

\item
All duals  of \eqref{coherence}--\eqref{end} hold (with the evident meaning of duals). 
\end{enumerate}
Since there are only finitely many maximal consistent sets, clusters, and bi-traces, there are only finitely many biboundaries.  
A biboundary is \defnn{closed} if it is 
defined on $N, S, E$, and $W$ (hence also on $b, l, r$, and $t$).
\end{definition}

Let $s$ be a surrectangle. The biboundary $\partial^s$ determined by $\partial$ is defined in the obvious way.

We now define three types of operations on biboundaries: joins, limits, and shuffles. See \Cref{fig:limit} for visual representations of limits and shuffles.

\begin{definition}
A biboundary $\partial$ is the \defnn{vertical join} of biboundaries $\partial_1$ and $\partial_2$, written $\partial_1 \join_-\partial_2$, if
\begin{itemize}
\item
$\partial_1(N)$ and $\partial_2(S)$ are both defined and are equal,

\item  either 
$\partial_1(W), \; \partial_2(W)$, and $\partial(W)$ are  all defined, $\partial_1(l)=\partial_2(b)$, and $\partial(W)=\partial_1(W)+\partial_1(l)+\partial_2(W)$, or $\partial_1(W), \partial_2(W)$ and $\partial(W)$ are all undefined; similarly for $E$,
\item
$\partial$ agrees with $\partial_1$ on $b, S, r$, and $-$, and with $\partial_2$ on $l, N, t$, and $+$.

\end{itemize}
The diagonal-dual concept is a \defnn{horizontal join}, written $\partial_1 \join_| \partial_2$.\footnote{The $-$ and $|$ subscripts indicate the orientation of the shared edge.}
\end{definition}

\begin{definition}
A biboundary $\partial^*$ is the \defnn{southeastern limit} of a biboundary $\partial_0$ using biboundaries $\partial_1, \partial_2, \partial_3$ if 
\begin{itemize}
\item $\partial_0=(\partial_2\join_|\partial_3)\join_-(\partial_0\join_|\partial_1)$,

\item
 the lower cluster of $\partial_1(E)$ is constantly $\partial_0(+)$,

\item
 the upper cluster of  $\partial_2(S)$ is constantly $\partial_0(-)$, \/

\item $\partial^*$ agrees with  $\partial_0$ over $\set{-, +, l, W, N}$, \/

\item
 $\partial^*(S)$, if defined, is a bi-trace where the upper cluster is constantly $\partial_0(-)$,

\item
 $\partial^*(E)$, if defined, is a bi-trace where the lower cluster is constantly $\partial_0(+)$.
\end{itemize}
 A \defnn{northwestern limit} is defined dually. 

 If $\Delta$ is a set of biboundaries and there are $\partial_0, \partial_1, \partial_2, \partial_3\in\Delta$ such that $\partial^*$ is the southeastern limit of $\partial_0$ using $\partial_1, \partial_2, \partial_3$, then $\partial^*$ is a southeastern limit over $\Delta$. Northwestern limits are dual.
We say that $\partial^*$ is a limit over $\Delta$ if it is either a southeastern or northwestern limit over $\Delta$.
\end{definition}

\begin{definition}\label{shuffle}
Let $\Delta$ be a collection of \emph{closed} biboundaries. The biboundary $\partial'$ is a \defnn{shuffle} of $\Delta$ if there is a \emph{nonempty} set $M \subseteq \mcs$ such that
\begin{enumerate}
\item\label{equal}
if $\partial'(W)$ is defined, then all upper clusters of $\partial'(W)$ equal $\partial'(-)$,

\item\label{shuffle_defects}
every future defect $\F \psi$ of $\partial'(-)$ is passed up to some $m \in M$, or there is a $\partial \in \Delta$ such that $\F \psi$ is passed up to either $\partial(b), \partial(l)$, $\partial(r)$, the interpolant of some cluster pair of $\partial(W)$ or $\partial(S)$, or some transition value of $\partial(W)$ or $\partial(S)$,

\item\label{fits}
for all $\partial \in \Delta$, we have $\partial(t) \leq \partial'(+)$, all future defects of $\partial(t)$ are passed up to $\partial'(+)$, and all upper clusters of $\partial(N)$ equal $\partial'(+)$,

\item\label{one-point}
for all $m \in M$, we have $m \leq \partial'(+)$, and all future defects of every $m$ are passed up to $\partial'(+)$,

\item
all duals of \eqref{equal}, \eqref{shuffle_defects}, \eqref{fits}, and \eqref{one-point} hold.
\end{enumerate}
\end{definition}

Now we are ready to define the biboundaries that we proceed to show are precisely those obtained from surrectangles.

\begin{definition}
A \defnn{ground fabricated biboundary} is a biboundary $\partial$ such that
\begin{enumerate}
\item
$\partial(-) = \partial(+)$,

\item\label{nonempty}
if $\partial(N)$ is defined, then all lower clusters of $\partial(N)$ equal $\partial(+)$,

\item
all duals of \eqref{nonempty} hold.
\end{enumerate}
A \defnn{fabricated biboundary} is either a ground fabricated biboundary, or a biboundary obtained recursively as the join, limit, or shuffle of fabricated biboundaries.
\end{definition}

\section{From fabricated biboundaries to surrectangles}\label{from_biboundary}

In this section we show that every fabricated biboundary is the biboundary obtained from some surrectangle, by describing how to construct such a surrectangle from a given biboundary. We use the recursive structure of fabricated biboundaries as given by their definition.

When we say a function $f$ fills $X$ \defnn{densely} with $M$ we mean that $f(\vec x) \in M$ for all $\vec x \in X$, and for each $m \in M$ the set $f^{-1}(m)$ is dense in $X$. It is clear that if $\mathcal U$ is an open subset of $\mathbb R^2$ and $c$ is a cluster, then there exists $f$ that fills $\mathcal U$ densely with $c$  (and this remains true when restrictions are placed on the behaviour of $f$ outside of $\mathcal U$). 
Similarly when $\mathcal U$ is an open line segment. Further, if for a biboundary $\partial$ we have a surrectangle satisfying $\partial^s = \partial$, we may assume that for every closed edge $e$ of $s$ and each associated cluster pair $c_i^-, c_i^+$ between transition points $\vec x_i, \vec x_{i+1}$, the core map of $s$ fills $(\vec x_i, \vec x_{i+1})$ densely with the interpolant of $c_i^-$ and $c_i^+$. Hence if surrectangles $s_1$ and $s_2$ have a common edge $e$, on which we obtain the same bi-trace and transition points from both surrectangles, then we may assume $s_1$ and $s_2$ agree on $e$.

\begin{lemma}\label{transformation}
Let $s$ be a surrectangle with core map $f$ and let $g_1$ and $g_2$ be order-preserving bijections $\mathbb R \to \mathbb R$. Then there is a surrectangle $s'$ with core map $(x, y) \mapsto f(g_1(x), g_2(y))$, the same supplementary clusters and transition values as $s$, and a transition point $(g_1^{-1}(x), g_2^{-1}(y))$ for every transition point $(x, y)$ of $s$. Moreover, $s'$ yields the same biboundary as $s$.
\end{lemma}

The proof of Lemma~\ref{transformation} is routine, and omitted.
The proofs of the following four  lemmas may be found in the appendix. Figure~\ref{fig:limit} illustrates the proofs of the last two.
\begin{lemma}[ground biboundaries]\label{closed:ground}
Let $\partial$ be a ground fabricated biboundary. Then there exists a surrectangle $s$ such that $\partial^s = \partial$.
\end{lemma}

\begin{lemma}[joins]\label{closed:join}
Let $\partial_1$ and $\partial_2$ be biboundaries such that the vertical join $\partial_1 \join_-\partial_2$ exists, and suppose there exist surrectangles $s_1$ and $s_2$ with $\partial^{s_1} = \partial_1$ and $\partial^{s_2} = \partial_2$. Then there exists a surrectangle $s$ such that $\partial^s = \partial_1 \join_-\partial_2$.

Similarly for horizontal joins.
\end{lemma}

\begin{lemma}[limits]\label{closed:limit}
Let $\partial^*$ be a southeastern limit of $\partial_0$ using $\partial_1, \partial_2, \partial_3$, and suppose there are surrectangles $s_0, s_1, s_2, s_3$ such that $\partial^{s_i}=\partial_i$, for $i<4$.  Then there exists a surrectangle $s^*$ such that  $\partial^{s^*} = \partial^*$. Similarly for  northwestern limits.
\end{lemma}

\begin{figure}[H]
\centering
\begin{tikzpicture}[scale=.5]
\draw(0, 0) --(0, 8)--(8,8)--(8, 0)--(0, 0);

\draw[dashed] (0, 4)--(4,4)--(4, 8);

\node at (2, 6) {$s_0$};

\draw[dashed] (0,2)--(4, 2)--(4, 4); \node at (2, 3) {$s_2$};
\draw[dashed] (4,4)--(6, 4)--(6, 8)--(4, 8); \node at (5, 6) {$s_1$};
\draw[dashed] (4, 2)--(6, 2)--(6, 4); \node at (5, 3) {$s_3$};

\draw[dashed] (0, 1)--(6,1)--(6,2); \node at (3, 1.5) {$s_2$};
\draw[dashed] (6,1)--(7,1)--(7, 8)--(6, 8);\draw[dashed](6, 2)--(7, 2);
\node at (6.5, 5) {$s_1$};
\node at (6.5, 1.5) {$s_3$};

\draw (0, 4)--(0,0)--(8, 0)--(8, 8)--(4, 8);

\node[above] at (0, 8) {$\vec l$};
\node[below] at (8, 0){$\vec r$};

\draw[dotted, very thick] (7.1,.9)--(7.9,0.1);

\node at (.9, .3) {$\partial_0(-)$};
\node at (5, .3) {$\partial_0(-)$};
\node[rotate=90] at(7.7, 7){$\partial_0(+)$};
\node [rotate =90]at(7.7, 3){$\partial_0(+)$};

\node at (4, -1) {(a)};
\end{tikzpicture}
\hspace{.5in}
\begin{tikzpicture}[scale=.5]
\newcommand\Square[1]{+(-#1,-#1) rectangle +(#1,#1)}
\draw (0,0) -- (8,0) --(8,8) --(0,8) -- (0, 0);
\node[above] at (0, 8) {$\vec l$};
\node[below] at (8, 0){$\vec r$};

\draw (1.2,5.2) rectangle (2.8,6.8);

\node at (2, 6) {$s_{\partial_0}$};

\draw(3.1, 3) rectangle (5, 4.9);
\node at (4, 4) {$s_{\partial_1}$};

\draw (5.8,.8) rectangle (7.2, 2.2);
\node at (6.5, 1.5) {$s_{\partial_0}$};

\draw [fill] (.5,7.5) circle [radius=1.5pt];
\draw [fill] (1,7) circle [radius=1.5pt];
\draw [fill] (2.9,5.1) circle [radius=1.5pt];
\draw [fill] (5.5,2.5) circle [radius=1.5pt];
\draw [fill] (5.3,2.7) circle [radius=1.5pt];
\draw [fill] (5.1,2.9) circle [radius=1.5pt];
\node  [right] at   (5.5,2.5)  {$m$};
\draw [fill] (7.5,.5) circle [radius=1.5pt];

\node at (2,2) {$\partial(-)$};
\node at (6,6) {$\partial(+)$};

\node at (4, -1) {(b)};

\end{tikzpicture}
\vspace{-5pt}
\caption{(a) A surrectangle for a southeastern limit of $\partial_0$ using $\partial_1, \partial_2, \partial_3$, and (b) a surrectangle for a shuffle of $\partial_0, \partial_1,\ldots $\label{fig:limit}}
\end{figure}
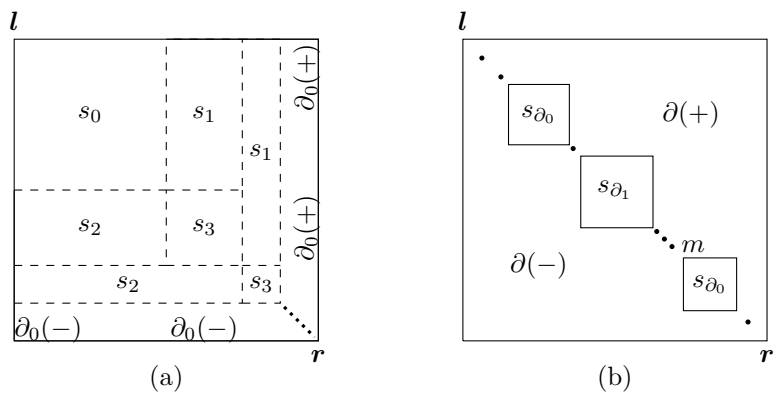

\begin{lemma}[shuffles]\label{closed:shuffle}
Let $\partial'$ be a shuffle of $\Delta$, and suppose that for all $\partial \in \Delta$ there is a surrectangle $s_\partial$ with $\partial^{s_\partial} = \partial$. Then there is a surrectangle $s'$ with $\partial^{s'} = \partial'$.
\end{lemma}

By Lemma~\ref{closed:ground}, the set $\Delta$ of biboundaries that can be obtained from surrectangles contains the ground fabricated biboundaries. By Lemma~\ref{closed:join}, Lemma~\ref{closed:limit}, and Lemma~\ref{closed:shuffle}, it is closed under joins, limits, and shuffles. Hence $\Delta$ contains all fabricated biboundaries. This proves, as promised, the following.

\begin{lemma}\label{section4summary}
Let $\partial$ be a fabricated biboundary. Then there exists a surrectangle $s$ such that $\partial^s = \partial$.
\end{lemma}

\section{From surrectangles to fabricated biboundaries}\label{from_surrectangle}

In this section we show that every biboundary obtained from a surrectangle is a fabricated biboundary (Lemma~\ref{thm:surr}). We do this by induction on the height of the surrectangle.

If a surrectangle $s$ has height $0$, then the upper and lower clusters of $\partial^s$ are equal, and from that it is easy to see that all conditions for $\partial^s$ to be a ground fabricated biboundary are satisfied.

In the remainder of this section, we assume that for every surrectangle $s$ of height no greater than $N$, the biboundary $\partial^s$ is fabricated, and aim to prove the statement for all surrectangles of height $N+1$.
First, note that by Lemma~\ref{transformation}, we only need prove the result for surrectangles with bounded domains. So henceforth we take bounded domains as an assumption. 

Let $s$ be a surrectangle with domain $R$, core map $f$, northwest corner $\vec l$, southeast corner $\vec r$, lower cluster $c^-$, and upper cluster $c^+$.  Let $I^-=f^{-1}(c^-)$ be the subset of $R$ that maps to the lower cluster, and $I^+=f^{-1}(c^+)$. Define $\Gamma^-$ to be the \emph{closure} (in $\mathbb R^2$) of the intersection of the boundary of $I^-$ with  the interior of $R$. Define $\Gamma^+$ similarly. Elementary topology shows that $\Gamma^-$ and $\Gamma^+$ are homeomorphic to closed (proper) line segments, meeting the boundary of $R$ only at their endpoints, and linearly ordered by $\triangleleft$; see \cite[Lemmas 2.10 and 2.11]{hrminkowski} for details. The notation $s\restr{[\vec x, \vec y]}$ signifies the surrectangle formed in the obvious way by restriction of $s$ to $[\vec x, \vec y] \cap R$ (assuming this is nondegenerate). It is straightforward to check that $s\restr{[\vec x, \vec y]}$ is indeed a surrectangle.

As before, proofs of lemmas are contained in the appendix.

\begin{lemma}\label{lem:disjoint}
If $\Gamma^-$ and $\Gamma^+$ are disjoint then $\partial^s$ is fabricated.
\end{lemma}

If $\Gamma^-$ and $\Gamma^+$ are \emph{not} disjoint, define a binary relation $\equiv$ over $R\setminus ({\down{I^-}}\cup {\up {I^+}})$ by letting $\vec x\equiv  \vec y\iff$  for \emph{all} nondegenerate rectangles $[\vec w, \vec z ] \subseteq [\vec x\wedge \vec y, \vec x\vee \vec y]$ the biboundary $\partial^{s\restr{[\vec w, \vec z ]}}$ is fabricated---clearly reflexive and symmetric, and also transitive (use a join of four biboundaries and the induction hypothesis), so an equivalence relation. By using joins and the induction hypothesis if necessary, we may assume $\vec l, \vec r \in R\setminus ({\down{I^-}}\cup {\up {I^+}})$. We aim to show that $\vec l\equiv\vec r$.

\begin{lemma}\label{lem:converge}
Let $\vec x_0\triangleleft \vec x_1\triangleleft \ldots \in R\setminus ({\down{I^-}}\cup {\up {I^+}})$ be an infinite sequence converging to $\vec x$.  If for all $i$ we have $\vec y\equiv \vec x_i$ then $\vec y\equiv \vec x$.
\end{lemma}  
  Let $P=\set{\vec l, \vec r}\cup(\Gamma^- \cap\Gamma^+)$---a closed set, inheriting a linear order $\triangleleft$ from $\Gamma^-$, and a subset of $R\setminus ({\down{I^-}}\cup {\up {I^+}})$.
Define a binary relation $\approx$ over $P$ as the smallest equivalence relation such that
\begin{enumerate}[(a)]
\item\label{a} $\vec p\approx \vec q$ whenever $\vec p$ is an immediate $\triangleleft$-successor of $\vec q$,
\item $\vec p\approx \vec q$ whenever $[\vec p\wedge \vec q, \vec p\vee \vec q]$ is degenerate,
\item\label{c}
all equivalence classes are topologically closed (in $\mathbb R^2$, equivalently, in $P$). 
\end{enumerate}

\begin{lemma}\label{lem:equiv} $\vec p\approx \vec q$ implies $\vec p\equiv \vec q$.
\end{lemma}
For any $\approx$-equivalence class $E$, let $l(E), r(E)$ be the extreme points with respect to $\triangleleft$ (equal if $E$ is a singleton). As $E$ is closed, $l(E), r(E) \in E$.  Write $R(E)$ for the closed rectangle $[l(E)\wedge r(E),  l(E) \vee r(E)]$ (a singleton if and only if $E$ is a singleton).  Since $\vec p\approx \vec q$ we know $l(E)\equiv r(E)$.

\begin{lemma}\label{lem:approx}
Either all elements of $P$ are $\approx$-equivalent, or there are uncountably many  singleton equivalence classes of $\approx$.
\end{lemma}

 Let $\vec x=r(E)$, where $E$ is the $\approx$-equivalence class of $\vec l$, and let $\vec y=l(E')$ where $E'$ is the $\approx$-equivalence class of $\vec r$.  Since $\vec x$ is the most southeastern point in $E$, and has no immediate $\triangleleft$-successor in $P$, there is no point of $\Gamma^-$ due south of $\vec x$, similarly no point of $\Gamma^+$ due east of $\vec x$. Dual conditions hold for $\vec y$.
  It follows that $f^+$ is constantly $c^-$ on the south and west edges of $[\vec x\wedge \vec y, \vec x\vee \vec y]$, and $f^-$ is constantly $c^+$ on the north and east edges. 
    
   Since $\vec l\approx \vec x$ and $\vec y\approx\vec r$ we know that $\vec l\equiv \vec x$ and $\vec y\equiv\vec r$.  The next lemma shows that $\vec x\equiv \vec y$.  We omit the proof, which is to check each of the conditions of Definition~\ref{shuffle}.
\begin{lemma}\label{lem:upper}
Suppose the upper cluster on the south and west edges of $s$ is constantly $c^-$, and the lower cluster on the north and east edges of $s$ is constantly $c^+$.
Then $\partial^s$ is a shuffle of $\partial^{s\restr{R(E)}}$ where $E$ ranges over $\approx$-equivalence classes.
\end{lemma}
Using joins and the induction hypothesis, this proves $\vec l\equiv\vec r$.  Hence we have obtained our goal.

\begin{lemma}\label{thm:surr}
Let $s$ be any surrectangle.  The biboundary $\partial^s$ is fabricated.
\end{lemma}

Combining Lemmas~\ref{section4summary} and \ref{thm:surr} we have the following.

\begin{lemma}\label{main_lemma}
A biboundary $\partial$ is of the form $\partial^s$ for some surrectangle $s$ if and only if $\partial$ is a fabricated biboundary.
\end{lemma}

Now we are in a position to prove our first main result.

\begin{theorem}\label{thm:decidable}
It is decidable whether a formula of the basic temporal language is valid on the frame consisting of two-dimensional Minkowski spacetime equipped with the irreflexive slower-than-light accessibility relation. The same is true with reflexive slower-than-light accessibility.
\end{theorem}

\begin{proof}
Decidability of validity is equivalent to decidability of satisfiability; we prove the latter.

 We first show that satisfiability of $\phi$ on $(\mathbb R^2, \prec)$ is equivalent to the existence of an open surrectangle having some point assigned a maximal consistent set containing $\phi$. An open surrectangle consists only of its core map $f$. Given a valuation $v$ on $(\mathbb R^2, \prec)$ and a point $\vec x$ at which $\phi$ holds, define $f$ by $f(\vec y) = \{\psi \in \Cl(\phi) \mid (\mathbb R^2, \prec), v, \vec y \models \psi \}$, and then $f$ will be a surrectangle on $\mathbb R^2$ with $\phi \in f(\vec x)$. Conversely, given such an $f$, for any propositional variable $p$ if $p$ appears in $\phi$, define $v(p) = \{\vec y \mid p \in f(\vec y)\}$, and otherwise define $v(p)$ arbitrarily. Then $\phi$ holds at $\vec x$ under valuation $v$.

Let $NE(x, y) = \{(x', y') \in \mathbb R^2 \mid x \leq x'\text{ and }y \leq y'\}$ and define $NW(\vec x)$, $SE(\vec x)$, and $SW(\vec x)$ similarly, in the evident way. The existence of an open surrectangle having some point $\vec x$ assigned a maximal consistent set containing $\phi$ is equivalent to the existence of four surrectangles with domains $NE(\vec x)$, $NW(\vec x)$, $SE(\vec x)$, and $SW(\vec x)$, agreeing on their shared edges and at $\vec x$, and with $\phi$ contained in the maximal consistent set assigned to $\vec x$. By Lemma~\ref{main_lemma}, this is in turn equivalent to the existence of four fabricated biboundaries of the appropriate types that match up in the appropriate way.

Since satisfiability of $\phi$ is equivalent to the existence of a set of four fabricated biboundaries with properties that are easily checked, and the set of all fabricated biboundaries is finite and computable, the satisfiability problem is decidable. This completes the proof of the irreflexive case.

The reflexive case follows by the reduction given by recursively replacing subformulas of $\phi$ of the form $\F\varphi$ with $\varphi \wedge \F \varphi$, and similarly for $\P$.
\end{proof}

\section{A PSPACE procedure for fabricated biboundaries}\label{section_pspace}

In this section, we refine the decidability results of Theorem~\ref{thm:decidable} to show the validity problems are $\mathsf{PSPACE}$-complete.

\begin{theorem}\label{pspace}
On the frame consisting of two-dimensional Minkowski spacetime equipped with the irreflexive slower-than-light accessibility relation, the set of validities of the basic temporal language is $\mathsf{PSPACE}$-complete. The same is true with reflexive slower-than-light accessibility.
\end{theorem}

As mentioned in the introduction, for the reflexive frame, the validities of the purely \emph{modal} fragment of the basic temporal language form {\bf S4.2}, and for the irreflexive frame, {\bf OI.2}. These are both known to be $\mathsf{PSPACE}$-complete \cite{Shapirovsky2005-SHAOPI}, so the validity problems for the \emph{entire} basic temporal language are $\mathsf{PSPACE}$-hard.

We provide a nondeterministic polynomial space algorithm for satisfiability, for the irreflexive frame. Hence validity is in $\mathsf{coNPSPACE}$. By the Immerman--Szelepcs\'enyi theorem \cite{immerman88,Szelepcsenyi1988}, $\mathsf{coNPSPACE} = \mathsf{NPSPACE}$, and by Savitch's theorem \cite{SAVITCH1970177}, $\mathsf{NPSPACE} = \mathsf{PSPACE}$, giving the result. The reflexive case follows by the same reduction as before.

Throughout this section, let the length of $\phi$ be $n$. By structural induction, the length of a formula bounds the number of its subformulas, so the cardinality of $\Cl(\phi)$ is linear in $n$. Hence any maximal consistent set---a subset of $\Cl(\phi)$---can be stored using a linear number of bits, and $\abs\mcs$ is at most exponential in $n$. All pertinent information about any cluster $c$ can also be stored in a linear number of bits, for we only need record $\{\psi \in \Cl(\phi) \mid \exists m \in c : \psi \in m\}$. The maximal length of a chain of distinct clusters or irreflexive members of $\mcs$ is also linear in $n$. Hence any biboundary can be stored using a quadratic number of bits, and the number of biboundaries is exponential in $n$.

Having nondeterministically chosen a bit string representing some $m \subseteq \Cl(\phi)$---a putative maximal consistent set---the conjunct over all formulas in $m$  is quadratic in $n$, so its satisfiability can be determined in polynomial space \cite{Spaan1993}. Hence we can determine if a chosen bit string represents a maximal consistent set using polynomial space. 

We can similarly nondeterministically `guess' sets of the form $\{\psi \in \Cl(\phi) \mid \exists m \in c : \psi \in m\}$ for some cluster $c$, using polynomial space. To do this, first guess a maximal consistent set $m$ and check it is reflexive; keep $m$ in memory. Then one-by-one for each bit string representing some $m' \subseteq \Cl(\phi)$, check $m' \in \mcs$ and that $m \lesssim m' \lesssim m$. If so, add all elements of $m'$ to an ongoing collection of formulas, discarding $m'$ after each iteration.

From the preceding discussion, it is clear we can also determine if a chosen string of bits represents a \emph{biboundary} using polynomial space.

The algorithm for satisfiability of $\phi$ first nondeterministically chooses four bit strings and checks they represent four compatible biboundaries, with $\phi$ at the appropriate corners---performed using polynomial space.  Then for each of these in turn, it is checked whether the biboundary is fabricated. The remainder of the proof is devoted to showing that this check can be performed using polynomial space.
 The procedure to do this is shown in \Cref{alg:fabricated}.

\begin{algorithm}[H]
\caption{Nondeterministic procedure to decide whether $\partial$ is fabricated}\label{alg:fabricated}
\begin{algorithmic}
\Procedure{fabricated}{$\partial$}
{\bf choose} either

\State {\bf option} 0
\State \hspace*{\algorithmicindent}{\bf check} $\partial$ is ground;

\State {\bf option} 1
\State \hspace*{\algorithmicindent}{\bf choose} $\partial_1, \partial_2$; {\bf check} they are biboundaries

\State \hspace*{\algorithmicindent}{\bf check} their join ({\bf choose} some direction) is $\partial$; {\bf release} $\partial$

\State \hspace*{\algorithmicindent}{\bf check}  \Call {fabricated}{$\partial_1$}; {\bf tail-call}  \Call {fabricated}{$\partial_2$}

\State {\bf option} 2
\State \hspace*{\algorithmicindent}{\bf choose} $\partial_1, \partial_2, \partial_3, \partial_4$; {\bf check} they are biboundaries

\State \hspace*{\algorithmicindent}{\bf check} $\partial$ is the limit ({\bf choose} direction) of $\partial_1$ using $\partial_2, \partial_3, \partial_4$; {\bf release} $\partial$

\State \hspace*{\algorithmicindent}{\bf check}  \Call {fabricated}{$\partial_2}$,  \Call {fabricated}{$\partial_3$}; {\bf release} $\partial_2, \partial_3$; {\bf check}  \Call {fabricated}{$\partial_4$}; {\bf tail-call}  \Call {fabricated}{$\partial_1$}

\State {\bf option} 3
\State \hspace*{\algorithmicindent}{\bf choose} $k \in \{0, 1, \dots, n\}$, $\partial_1, \dots, \partial_k$; {\bf check} they are biboundaries

\State \hspace*{\algorithmicindent}{\bf choose} $m_1, \dots, m_n$; {\bf check} they are maximal consistent sets

\State \hspace*{\algorithmicindent}{\bf check} $\partial$ is the shuffle of $\partial_1, \dots, \partial_k$ using $m_1, \dots, m_n$; {\bf release} $\partial$

\State\hspace*{\algorithmicindent}{\bf for} $i=1, \dots, k$ {\bf do}

\State \hspace*{\algorithmicindent}\hspace*{\algorithmicindent}{\bf check}  \Call {fabricated}{$\partial_i$}
\State\hspace*{\algorithmicindent}{\bf end for}
\EndProcedure
\end{algorithmic}
\end{algorithm}

We assume that during execution of \Cref{alg:fabricated}, the formula $\phi$  is a global constant (and therefore is $n$ too). All choices are made nondeterministically. At a {\bf check} the algorithm fails if the check fails, otherwise it proceeds. A {\bf tail-call} uses tail recursion. The algorithm succeeds if it terminates without failure.

Clearly \Cref{alg:fabricated} can only succeed for a biboundary $\partial$ if $\partial$ is fabricated. Let $S(N)$ be the amount of space required by \Cref{alg:fabricated} to succeed for any fabricated biboundary of height $N$. We prove that $S(N) = \mathcal O((N+1) n^3 )$ (as a function of $N$ and $n$) by induction on $N$. Hence \Cref{alg:fabricated} requires space $\mathcal O(n^4)$ to check any biboundary.

 If $\partial$ is a fabricated biboundary of height $0$, then it is ground, so success can be achieved by taking the first branch, requiring space $\mathcal O(n^2)$.

 For the inductive step, we show that $S(N+1) \leq S(N) + \mathcal O(n^3)$. The proof has a similar structure to the proof in the previous section. Let $\partial$ be a fabricated biboundary of height $N+1$, and $s$ be a surrectangle such that $\partial^s = \partial$. Define $\Gamma^-$, $\Gamma^+$, $\vec l$, and $\vec r$ as in the previous section. Let $\delta$ be the space required to store any biboundary. Thus $\delta$ is a quadratic function of $n$, independent of $N$.

As always, proofs of lemmas are relegated to the appendix.

\begin{lemma}\label{lem:delta1}
 If $\Gamma^-$ and $\Gamma^+$ are disjoint, the space required is bounded by $S(N) + \delta$.
\end{lemma}

\begin{lemma}\label{lem:delta3}
If $\Gamma^-$ and $\Gamma^+$ intersect only on the boundary of $s$, the space required is bounded by $S(N) + 3  \delta$. 
\end{lemma}

As in the previous section, let $P = \set{\vec l, \vec r}\cup(\Gamma^- \cap\Gamma^+)$. Additionally, let $\Delta = \{\partial^{s\restr{[\vec x \wedge \vec y, \vec x \vee \vec y]}} \mid \vec x, \vec y \in P, \text{ and $\vec y$ is an immediate $\triangleleft$-successor of $\vec x$}\}$.

\begin{lemma}\label{shuffle_bound}
If the biboundary of $s$ is a shuffle of a subset $\Delta'$of $\Delta$, the space required is bounded by $S(N) + 3  \delta + n \delta$. 
\end{lemma}

If $\vec l$ and $\vec r$ are each the limit of elements of $P$ in the interior of $s$, then the biboundary of $s$ is the shuffle of $\Delta'$, using $M$, where $\Delta' = \{\partial^{s\restr{[\vec x \wedge \vec y, \vec x \vee \vec y]}} \mid \vec x, \vec y \in  P \setminus \{ \vec l, \vec r \}, \text{ and $\vec y$ is the successor of $\vec x$} \}$ and $M = P \setminus \{ \vec l, \vec r \}$.

Now suppose only one of $\vec l$, $\vec r$ is a limit of elements of $P$ in the interior of $s$. Say $\vec l$ is such a limit, but $\vec r$ is not. If $\vec r$ has a direct predecessor $\vec r'$, then $\vec r'$ is in the interior of $s$. If there are points of $P$ (strictly) due north of $\vec r$, let the northmost one be $\vec w$. Then either $\vec w$ is a limit, and using a join we can reduced to the case of the previous paragraph, or there is a point $\vec r'$ either an immediate predecessor or due west of $\vec w$. In the second case, $\vec r'$ is in the interior of $\vec s$. Similarly if there are points of $P$ due west of $\vec r$. In each case we obtain an $\vec r'$ in the interior of $P$ and the biboundary $\partial_{SE}$ of $s\restr{[\vec r \wedge \vec r', \vec r \vee \vec r']}$ can be checked in $S(N) + 3  \delta$, by Lemma~\ref{lem:delta3}.

We are going to obtain the biboundary of $s$ as the join of four biboundaries. In the southeast corner is $\partial_{SE}$. If the biboundary of $s\restr{[\vec l \wedge \vec r', \vec l \vee \vec r']}$ is of reduced height then we are done. Otherwise, the other three biboundaries we use are modifications of restrictions. For the northwest corner, take the biboundary of $s\restr{[\vec l \wedge \vec r', \vec l \vee \vec r']}$. Modify its south edge to the bi-trace $(c^-, c^-)$ and its east edge to $(c^+, c^+)$. This is still a biboundary; call it $\partial_{NW}$. For the northeast corner, take the biboundary of $s\restr{[\vec r', \vec t]}$ (which is ground). Modify its west edge to the bi-trace $(c^+, c^+)$. This is still a biboundary (and still ground). Similarly for the southwest corner.

The biboundary of $s$ is the join of these four biboundaries. It only remains to show $\partial_{NW}$ is fabricated, and that this can be checked efficiently. As $\partial_{NW}$ is of full height, its upper cluster is $c^+$ and lower $c^-$. By appealing to the facts that $s$ is a surrectangle and $s\restr{[\vec l \wedge \vec r', \vec l \vee \vec r']}$ is also a surrectangle, we can see that $\partial_{NW}$ is the shuffle of $\Delta'$, using $M$, where $\Delta' = \{\partial^{s\restr{[\vec x \wedge \vec y, \vec x \vee \vec y]}} \mid \vec x, \vec y \in (P \setminus \{ \vec l, \vec r \}) \cap [\vec l \wedge \vec r', \vec l \vee \vec r'], \text{ and $\vec y$ is the successor of $\vec x$} \}$, and $M =  (P \setminus \{ \vec l, \vec r \}) \cap [\vec l \wedge \vec r', \vec l \vee \vec r'] $. The set $\Delta'$ is a subset of $\Delta$, so by Lemma~\ref{shuffle_bound}, the biboundary $\partial_{NW}$ can be checked using no more than $S(N) + 3  \delta + n \delta$ space. Hence $\partial$ can been checked within $S(N) + \mathcal O(n^3)$, as promised.

The case where neither $\vec l$ nor $\vec r$ are limits is similar.

\section{Halpern--Shoham logic}\label{section_intervals}

In this section we explain how the $\mathsf{PSPACE}$ procedure for the temporal validities of $(\mathbb R^2, \prec)$ can be modified to procedures performing the same function for certain Halpern--Shoham logics of intervals on the real line.

In Halpern--Shoham logic, intervals are identified with pairs $(x, y)$ of points, with either $x < y$ (\defnn{strict interval semantics}) or $x \leq y$ (\defnn{non-strict semantics}).\footnote{Thus there is no distinction between open, half-open, and closed intervals, and unbounded intervals are not present.} There are thirteen different atomic relations that may hold between two strict intervals. Following Allen \cite{Allen:1983:MKT:182.358434}, we call these $\mathtt{equals}$, $\mathtt{before}$, $\mathtt{after}$, $\mathtt{during}$, $\mathtt{contains}$, $\mathtt{starts}$, $\mathtt{started\_by}$, $\mathtt{finishes}$, $\mathtt{finished\_by}$, $\mathtt{meets}$, $\mathtt{met\_by}$, $\mathtt{overlaps}$, and $\mathtt{overlapped\_by}$.\footnote{When point-intervals are present, pairs of intervals of the form $((x, x), (x, y))$ for $x<y$ are considered by Halpern and Shoham to stand in the relation $\mathtt{starts}$ (and not in $\mathtt{meets}$). Similarly for $\mathtt{started\_by}$, $\mathtt{finishes}$, and $\mathtt{finished\_by}$.} Modalities that may be included in a Halpern--Shoham logic are any corresponding to a relation given by the union of some of these thirteen.

Let $H_<$ be the open half-plane $\{(x, y) \in \mathbb R^2 \mid x < y\}$. The frame $(H_<, \prec)$ is precisely the frame of strict intervals of $\mathbb R$ with the relation $\mathtt{overlaps} \cup \mathtt{meets} \cup \mathtt{before}$. Hence the temporal logic of $(H, \prec)$ is the strict Halpern--Shoham logic of $\mathbb R$ with two modalities corresponding to this relation and its converse.  
As in \cite{Goldblatt1980}, it can be shown that an arbitrary finitely generated directed partial order is a $p$-morphic image of  the reflexive closure of $(H_<,\prec)$. Hence an arbitrary  purely modal formula (not using the past modality)  is in {\bf S4.2} if and only if it is valid for the reflexive closure of $(H_<, \prec)$,
and hence the validity of temporal formulas  over this frame is {\sf PSPACE}-hard. The $\mathsf{PSPACE}$-hardness of irreflexive $(H_<, \prec)$ itself, follows.

A \defnn{surtriangle} is similar to a surrectangle, but the domain of the core map $f$ is either $\set{(x, y) \in H_< \mid  a\leq x,\; y\leq  b}$, for some $a<b$, or a similar set where either $a\leq x$ is replaced by $a<x$, or $y\leq b$ is replaced by ${y<b}$, or both. A finite sequence of upper supplementary clusters and transition points are defined along $\set{(x, b) \in H_< \mid a<x<b}$ if contained in the domain of $f$, and similarly for $\set{(a, y) \in H_< \mid a<y<b}$ with lower supplementary clusters. 

  A biboundary is called 
\defnn{triangular} if its domain is disjoint from $\set{b, t, r,\allowbreak S, E}$.  A surtriangle determines a triangular biboundary in the obvious way.    If $\tau_1$ is a triangular biboundary containing its northern edge, $\tau_2$ a  triangular biboundary containing its western edge, and  $\partial$ a (rectangular) biboundary closed on its southern and eastern edges,
such that $\tau_1(N)=\partial(S)$ and $\partial(W)=\tau_2(E)$, then the join $J(\tau_1, \partial, \tau_2)$ is the triangular biboundary formed by joining the three parts together. (The inconsequential southeast corner of $\partial$ is discarded.)  A triangular biboundary is \defnn{fabricated} if it is either ground, or the join of a ground triangular biboundary, a fabricated rectangular biboundary, and a triangular biboundary of strictly smaller depth.  The proofs of the following  lemmas and theorem are similar to the proofs of Lemma~\ref{main_lemma}, Theorems~\ref{thm:decidable}, and Theorem~\ref{pspace}.
\begin{lemma}
A triangular biboundary is fabricated if and only if it is the triangular biboundary of some surtriangle.
\end{lemma}
\begin{lemma}
There is a {\sf PSPACE} algorithm to determine whether a triangular biboundary is fabricated.
\end{lemma}

\begin{theorem}
Let $H_< = \{(x, y) \in \mathbb R^2 \mid x < y\}$, and let $R$ be either the relation given by $(x, y) R (x', y') \iff x < x' \text{ and } y < y'$ (so $R = \mathtt{overlaps} \cup \mathtt{meets} \cup \mathtt{before}$) or the reflexive closure of this relation (so $R=\mathtt{equals}\cup \mathtt{overlaps} \cup \mathtt{meets} \cup \mathtt{before}$). Then the temporal logic of $(H_<, R)$ is $\mathsf{PSPACE}$-complete.
\end{theorem}

\Appendix{}

\textbf{Proof of Lemma~\ref{closed:ground} (ground biboundaries)}
Let $R$ be the appropriate rectangle (given $\partial$) with $(0, 1) \times (0, 1) \subseteq R \subseteq [0, 1] \times [0,1]$. Define $f \from R \to \mcs$ as filling the interior of $R$ densely with the cluster $\partial(+)$. Define $f$ on any corners of $R$ in the evident way.

 If $\partial$ is defined on $N$, with $\partial(N) = (c_0^-, c_0^+, b_1, \dots, c_n^+)$ say, then define $f(\frac{i} {n+1}, 1)= b_i$ for each $i$, and fill each line segment $((\frac{i} {n+1}, 1), (\frac{i+1} {n+1}, 1))$ densely with the interpolant of $c_i^-$ and $c_i^+$. Define the supplementary clusters for this edge to be $c_0^+, \dots, c_n^+$, and the sequence of transition points to be $(\frac {1}{n+1},1), \dots, (\frac{n}{n+1}, 1)$. Similarly if $\partial$ is defined on $S, E$, or $W$.

The conditions on biboundaries in general, together with those for ground fabricated ones in particular, ensure our construction satisfies all conditions necessary to be a surrectangle.

\smallskip

\noindent\textbf{Proof of Lemma~\ref{closed:join} (joins)}  
We know the bi-trace obtained from the northern edge of $s_1$ equals that obtained from the southern edge of $s_2$, and in particular these edges have the same number of transition points. Then by Lemma~\ref{transformation}, we may assume these two edges and their cluster pairs, transition points, transition values, and values of $f$ at any end-points coincide. As we remarked in the opening of \Cref{from_biboundary}, this is sufficient to assume $s_1$ and $s_2$ agree on the common edge.

We define $s$ with domain the union of the domains of $s_1$ and $s_2$ in the obvious way. Then it is straightforward to check that $s$ forms a surrectangle and that $\partial^s = \partial_1 \join_-\partial_2$.

The proof for horizontal joins is completely analogous.

\smallskip

\noindent\textbf{Proof of Lemma~\ref{closed:limit} (limits)}  Since $\partial^*$ is a southeastern limit of $\partial_0$ using $\partial_1, \partial_2, \partial_3$, the vertical join of $\partial_2, \partial_0$ is defined, and similarly for other joins.  
\Cref{fig:limit}(a) illustrates how a surrectangle for $\partial^*$ can be constructed using transformed copies of the surrectangles $s_0, s_1, s_2, s_3$.  The domain of the surrectangle is a rectangle with interior $(0, 1)\times (0, 1)$.  
The upper-left quadrant $(0, \frac 12)\times (\frac 12, 1)$  together with boundaries dictated by $\partial_0$ is a copy of $s_0$.  
Each rectangle $( 0, 1-\frac{1}{2^k})\times (\frac{1}{2^{k+1}}, \frac{1}{2^k})$, plus appropriate boundaries, is a copy of $s_2$, for $k\geq 1$.     Use Lemma~\ref{transformation} to ensure agreement along common boundaries.
Transformed copies of $s_1$ are used dually.   
The square $[1-\frac{1}{2^k}, 1-\frac{1}{2^{k+1}}]\times[\frac{1}{2^{k+1}}, \frac{1}{2^k}]$ is a copy of $s_3$, for $k\geq 1$.  
This covers the interior of the unit square, and northern, western edges, and northwest corner, if included in the domain of $\partial^*$.  Use $\partial^*$ to define the corners $b, r, t$ if included in the domain of $\partial^*$. Use the interpolant of the lower and upper clusters of $\partial^*(S)$ densely along the southern edge of the unit square, if $S$ is included in the domain of $\partial^*$, similarly for the eastern edge.

The proof for northwestern limits is completely analogous.

\smallskip

\noindent\textbf{Proof of Lemma~\ref{closed:shuffle} (shuffles)} The closure of the rectangle of $s'$ will be $[0,1]\times[0, 1]$.   Let $M \subseteq \mcs$ be such that the conditions in Definition~\ref{shuffle}, the definition of a shuffle, are satisfied. Let $d$ be the open diagonal line-segment $((0, 1), (1, 0))$. 
We describe how to construct $s'$ in stages, as illustrated in \Cref{fig:limit}(b).  First fill the area of $(0, 1) \times (0,1)$ below $d$ densely with $\partial'(-)$, and fill the area above this diagonal densely with $\partial'(+)$.   Fill $d$ with any element of $M$.

For any edge that $\partial'$ indicates should be closed in $s'$, assign the appropriate sequence of supplementary clusters, and evenly spaced transition points. Assign the appropriate transition values at the transition points, and between transition points fill the edge densely with the appropriate interpolant. If $\partial'$ indicates any corners are required, assign them the appropriate maximal consistent set. 

Next, we successively modify the interior of the construction. Each point will be updated at most once, so this process has a well-defined limit. We maintain a finite set $S$ of disjoint open subsegments of $d$, initialised to $\set{d}$.  At a later stage, pick $d'\in S$ and $b\in\Delta \cup M$. If $b\in \Delta$, reassign the closed rectangle whose diagonal is the central third of $d'$, using the surrectangle $s_b$. Otherwise, reassign the midpoint $\vec m$ of $d'$ with the maximal consistent set $b$.  In either case, the segment $d'$ is replaced by two in $S$---in the first case the open initial and final thirds of $d'$, in the second, the two halves of $d' \setminus \{\vec m\}$.  Schedule the choices of $d'$ and $b$ so that for every segment $d'$ occurring in this construction, for every $b\in\Delta \cup M$, a choice $d^*, b$ is eventually selected, where $d^*$ is a subsegment of $d'$.  

The limit of this process gives a well-defined map $f$ from the appropriate subset of $[0,1]\times[0,1]$ to $\mcs$, and supplementary clusters where appropriate. We argue that together these form a surrectangle, $s'$.

The constraints placed on $\partial'$ by the ordering conditions in the definition of a biboundary, together with the ordering conditions in the definition of a shuffle, ensure  that $f$ is preorder preserving. Our construction has the property that every point below $d$ that did not undergo reassignment has an open subset of points in its future that also did not undergo reassignment, \emph{and} has a copy of each $s_\partial$, and each $m \in M$ in its future. Noting this, it is straightforward to check that all future defects are either resolved internally or passed upwards. The other conditions in the definition of a surrectangle are also straightforward to check.

 Clearly the biboundary of $s'$ is $\partial'$, as required.

\smallskip

\noindent\textbf{Proof of Lemma~\ref{lem:disjoint}} If $\Gamma^-$ and $\Gamma^+$ are disjoint, then by the boundedness of $R$, and since they are closed, they are bounded away from each other, with a bound $\varepsilon > 0$ say. Then we may divide $R$  into a \emph{finite} grid  of rectangles each with a diagonal shorter than $\varepsilon$.  The restriction of $s$ to such a rectangle has height $N$ or less, so yields a fabricated biboundary. Then $\partial^s$ is a join of such biboundaries, and is therefore itself fabricated.

\smallskip

\noindent\textbf{Proof of Lemma~\ref{lem:converge}}  
Using the inductive hypothesis, we may assume that the biboundary of $s\restr{[\vec y\wedge \vec x, \vec y\vee \vec x]}$ is of height $N+1$, that is, both $\Gamma^-$ and $\Gamma^+$ intersect the interior of $[\vec x \wedge \vec y, \vec x \vee \vec y]$.
 As each $\vec x_i$ is in $R\setminus (\down I^-\cup \up I^+)$, we know $s\restr{[\vec x_i, \vec x \vee \vec y]}$ (if $[\vec x_i, \vec x \vee \vec y]$ is nondegenerate) has height at most $N$, so its biboundary is fabricated, similarly for $s\restr{[\vec x \wedge \vec y, \vec x_i]}$.  Hence, using joins, it suffices to prove the biboundary of $s\restr{[\vec x_i\wedge\vec x, \vec x_i\vee\vec x]}$ is fabricated, for some $i$. Then it follows using Lemma~\ref{lem:disjoint} and the inductive hypothesis that we may assume $\Gamma^-$ and $\Gamma^+$ meet at $\vec x$, that no point due west of $\vec x$ is in $\Gamma^-$, and that no point due north of $\vec x$ is in $\Gamma^+$. Together, these assumptions imply the upper cluster on the southern edge of $s\restr{[\vec y\wedge \vec x, \vec y\vee \vec x]}$ is constantly $c^-$, and the lower cluster on the eastern edge is constantly $c^+$. 

Since there are only finitely many biboundaries, by taking a subsequence, we can assume the biboundary of ${s\restr{[\vec y\wedge \vec x_i, \vec y\vee \vec x_i]}}$ is constant, and  by our previous assumptions we may also assume that $\partial^{s\restr{[\vec x_0, \vec y\vee\vec  x_1]}}$ has constant lower cluster $\partial^s(+)$ on its eastern edge, and $\partial^{s\restr{[\vec y\wedge \vec x_1, \vec x_0]}}$ has constant upper cluster $\partial^s(-)$ on its southern edge. Then the biboundary of ${s\restr{[\vec y\wedge \vec x, \vec y\vee \vec x]}}$ is a southeastern limit of the biboundary of ${s\restr{[\vec y\wedge \vec x_0, \vec y\vee \vec x_0]}}$ using the biboundaries of ${s\restr{[\vec x_0, \vec y\vee\vec  x_1]}}, {s\restr{[\vec y\wedge \vec x_1, \vec x_0]}}$, and ${s\restr{[\vec x_0\wedge \vec x_1, \vec x_0\vee \vec x_1]}}$. All of these are fabricated, by the hypothesis that $\vec y \equiv \vec x_1$. Hence the biboundary of ${s\restr{[\vec y\wedge \vec x, \vec y\vee \vec x]}}$ is fabricated. Furthermore the biboundary of the restriction of $s$ to any nondegenerate $[\vec w, \vec z] \subsetneq [\vec y\wedge \vec x, \vec y\vee \vec x]$ is clearly fabricated, so $\vec y\equiv \vec x$.

\smallskip

\noindent\textbf{Proof of Lemma~\ref{lem:equiv}}   It is clear that if $[\vec p\wedge \vec q, \vec p\vee \vec q]$ is degenerate then $\vec p\equiv \vec q$.
If $\vec p$ is an immediate $\triangleleft$-successor of $\vec q$  pick $\vec y \triangleleft \vec x_0\triangleleft \vec x_1\triangleleft\ldots$ in the interior of $[\vec p\wedge \vec q, \vec p\vee \vec q]$, converging to $\vec p$.  Since $\vec y, \vec x_i$ are in the interior of the rectangle, the upper and lower clusters of $s\restr{[\vec y\wedge \vec x_i, \vec y\vee \vec x_i]}$ are bounded away, hence the biboundary is fabricated.  By Lemma~\ref{lem:converge} the biboundary of $s\restr{[\vec y\wedge \vec p, \vec y\vee \vec p]}$ is fabricated.  Similarly, by considering northwestern limits, the biboundary of $s\restr{[\vec p\wedge \vec q, \vec p\vee \vec q]}$ is fabricated and $\vec p\equiv \vec q$.  By Lemma~\ref{lem:converge} and its diagonal dual, the $\equiv$-equivalence classes are topologically closed. The lemma follows.

\smallskip

\noindent\textbf{Proof of Lemma~\ref{lem:approx}}  
The upper boundary $\Gamma^-$ of $I^-$ is homeomorphic to the closed unit interval $[0, 1]$, and such a homeomorphism $f$ will map $R(E)\cap \Gamma^-$ to an interval, for any $\approx$-equivalence class $E$. The set $I^-$ is the disjoint union of the  closed sets $R(E)\cap \Gamma^-$ as $E$ ranges over $\approx$-equivalence classes. Hence $\{f(R(E)\cap \Gamma^-) \mid E \in {P/{\approx}}\}$ is a partition of $[0, 1]$ into closed intervals. It follows that if there is more than one $\approx$-equivalence class, then uncountably many of these closed intervals are singletons; see \cite[Lemma~2.9]{hrminkowski}. For each singleton interval, the corresponding $\approx$-equivalence class must itself be a singleton.

\smallskip

\noindent\textbf{Proof of Lemma~\ref{lem:delta1}}
 If $\Gamma^-$ and $\Gamma^+$ are disjoint, they are bounded apart, by $\varepsilon$ say. Then either $\Gamma^-$ is bounded away from north edge of $s$ by $\frac{\varepsilon}{\sqrt2}$, or $\Gamma^+$ is bounded away from the \emph{west} edge by $\frac{\varepsilon}{\sqrt2}$. Without loss of generality, assume the latter. Let $s_1$ be the restriction of $s$ to the region at a distance no more than $\frac{\varepsilon}{\sqrt2}$ from the west edge of $s$, and $s_2$ be the restriction of $s$ to the region \emph{at least} $\frac{\varepsilon}{\sqrt2}$ from the west edge. Then $\partial = \partial^{s_1}\join_|\partial^{s_2}$, and $\partial^{s_1}$ has height no greater than $N$. Hence it is possible to choose  option 1 with $\partial^{s_1}$, $\partial^{s_2}$, and check $\partial_1$ is fabricated using no more than $S(N) + \delta$ space. The tail-recursive call is to the biboundary of $s_2$, which is smaller than $s$ in one dimension by $\frac{\varepsilon}{\sqrt2}$ and---if still of height $N+1$---has upper and lower clusters still bounded apart by $\varepsilon$. Hence iterating the described choice scheme requires space $S(N) + \delta$ each iteration, and must eventually result in a reduction of the height of the second biboundary $s_2$ in the join, after which the recursive call to $s_2$ can succeed with $S(N)$ space. The maximum space required is $S(N) + \delta$, as claimed.

\smallskip

\noindent\textbf{Proof of Lemma~\ref{lem:delta3}}
If $\Gamma^-$ and $\Gamma^+$ intersect only at \emph{one} point on the boundary of $s$, by Lemma~\ref{lem:delta1}, we may assume (without exceeding our space allowance) it is either $\vec l$ or $\vec r$---without loss of generality, assume $\vec r$. Then for some $\vec x$ and $\vec y$ in the interior of $s$, the biboundary $\partial$ is the limit of $\partial^{s\restr{[\vec l \wedge \vec x, \vec l \vee \vec x]}}$ using $\partial^{s\restr{[\vec x , \vec l \vee \vec y]}}$, $\partial^{s\restr{[\vec l \wedge \vec y,  \vec x]}}$, and $\partial^{s\restr{[\vec x \wedge \vec y, \vec x \vee \vec y]}}$, and both $\partial^{s\restr{[\vec x , \vec l \vee \vec y]}}$ and $\partial^{s\restr{[\vec l \wedge \vec y,  \vec x]}}$ are of height no greater than $N$. Whilst checking $\partial^{s\restr{[\vec x , \vec l \vee \vec y]}}$ and $\partial^{s\restr{[\vec l \wedge \vec y,  \vec x]}}$, at most $S(N)$ space is needed, plus $3\delta$ to store the other three biboundaries. Whilst checking $\partial^{s\restr{[\vec l \wedge \vec y,  \vec x]}}$, by Lemma~\ref{lem:delta1}, at most $S(N) + \delta$ is needed, plus $\delta$ to store $\partial^{s\restr{[\vec l \wedge \vec x, \vec l \vee \vec x]}}$. Lastly, $\partial^{s\restr{[\vec l \wedge \vec x, \vec l \vee \vec x]}}$ is checked, and $S(N) + \delta$ is sufficient by Lemma~\ref{lem:delta1}. The maximum space needed is bounded by $S(N) + 3\delta$.

If $\Gamma^-$ and $\Gamma^+$ intersect at \emph{two} points on the boundary of $s$, then using the same choice scheme, $\Gamma^- \cap [\vec l \wedge \vec x, \vec l \vee \vec x]$ and $\Gamma^+ \cap [\vec l \wedge \vec x, \vec l \vee \vec x]$ will intersect only on the boundary of $s\restr{[\vec l \wedge \vec x, \vec l \vee \vec x]}$, and at at most one point. Then by the previous case, checking the biboundary of $s\restr{[\vec l \wedge \vec x, \vec l \vee \vec x]}$ can be done with $S(N) + 3\delta$ space. As this is done last, the bound $S(N) + 3\delta$ remains intact.

\smallskip

\noindent\textbf{Proof of Lemma~\ref{shuffle_bound}}
If a biboundary $\partial'$ is a shuffle over $\Delta'$ using $M' \subseteq \mcs$ then it is the shuffle of any subset $\Delta''$ of $\Delta'$ using any subset $M''$ of $M'$, so long as any future defect of $\partial'(-)$ is passed up somehow to $\Delta''$ or $M''$, and any past defect of $\partial'(+)$ is passed down to $\Delta''$ or $M''$. There are at most $n$ defects, past or future, so any shuffle is the shuffle of at most $n$ biboundaries, using at most $n$ auxiliary maximal consistent sets. Hence by choosing option 3, \Cref{alg:fabricated} can succeed in checking $\partial'$. Since each of the up to $n$ subchecks requires no more than $S(N) + 3\delta$ space, checking $\partial'$ can be accomplished using $S(N) + 3\delta + n\delta$ space.

\bibliographystyle{aiml18}

\begin{thebibliography}{10}
\expandafter\ifx\csname url\endcsname\relax
  \def\url#1{\texttt{#1}}\fi
\expandafter\ifx\csname urlprefix\endcsname\relax\def\urlprefix{URL }\fi
\newcommand{\enquote}[1]{``#1''}

\bibitem{Allen:1983:MKT:182.358434}
Allen, J.~F., \emph{Maintaining knowledge about temporal intervals},
  Communications of the ACM \textbf{26} (1983), pp.~832--843.

\bibitem{10.1007/978-3-540-89439-1_41}
Bresolin, D., D.~Della~Monica, V.~Goranko, A.~Montanari and G.~Sciavicco,
  \emph{Decidable and undecidable fragments of {H}alpern and {S}hoham's
  interval temporal logic: Towards a complete classification}, in:
  I.~Cervesato, H.~Veith and A.~Voronkov, editors, \emph{proceedings of 15th
  conference on Logic for Programming, Artificial Intelligence, and Reasoning}
  (2008), pp. 590--604.

\bibitem{Bresolin:2017:HFH:3130378.3105909}
Bresolin, D., A.~Kurucz, E.~Mu\~{n}oz Velasco, V.~Ryzhikov, G.~Sciavicco and
  M.~Zakharyaschev, \emph{Horn fragments of the {H}alpern--{S}hoham interval
  temporal logic}, ACM Transactions on Computational Logic \textbf{18} (2017),
  pp.~22:1--22:39.

\bibitem{BMSS11}
Bresolin, D., A.~Montanari, P.~Sala and G.~Sciavicco, \emph{What's decidable
  about {H}alpern and {S}hoham's interval logic? {T}he maximal fragment
  {ABBL}}, in: \emph{proceedings of 26th IEEE symposium on Logic in Computer
  Science}, 2011, pp. 387--396.

\bibitem{Goldblatt1980}
Goldblatt, R., \emph{Diodorean modality in {M}inkowski spacetime}, Studia
  Logica \textbf{39} (1980), pp.~219--236.

\bibitem{Halpern:1991:PML:115234.115351}
Halpern, J.~Y. and Y.~Shoham, \emph{A propositional modal logic of time
  intervals}, Journal of the ACM \textbf{38} (1991), pp.~935--962.

\bibitem{hrminkowski}
Hirsch, R. and M.~Reynolds, \emph{The temporal logic of two dimensional
  {M}inkowski spacetime is decidable}, The Journal of Symbolic Logic  (2018),
  ({i}n press), arXiv:1507.04903.

\bibitem{immerman88}
Immerman, N., \emph{Nondeterministic space is closed under complementation},
  SIAM Journal on Computing \textbf{17}, pp.~935--938.

\bibitem{5970233}
Marcinkowski, J. and J.~Michaliszyn, \emph{The ultimate undecidability result
  for the {H}alpern--{S}hoham logic}, in: \emph{proceedings of 26th IEEE
  symposium on Logic in Computer Science}, 2011, pp. 377--386.

\bibitem{MM13}
Marcinkowski, J. and J.~Michaliszyn, \emph{The undecidability of the logic of
  subintervals}, Fundamenta Informaticae \textbf{XX} (2013), pp.~1--25.

\bibitem{6311113}
Montanari, A. and P.~Sala, \emph{An optimal tableau system for the logic of
  temporal neighborhood over the reals}, in: \emph{proceedings of 19th
  symposium on Temporal Representation and Reasoning}, 2012, pp. 39--46.

\bibitem{SAVITCH1970177}
Savitch, W.~J., \emph{Relationships between nondeterministic and deterministic
  tape complexities}, Journal of Computer and System Sciences \textbf{4}
  (1970), pp.~177--192.

\bibitem{Shapirovsky2005-SHAOPI}
Shapirovsky, I., \emph{On {PSPACE}-decidability in transitive modal logics},
  in: R.~Schmidt, I.~Pratt{-}Hartmann, M.~Reynolds and H.~Wansing, editors,
  \emph{Advances in Modal Logic, papers from the 5th conference, held in
  Manchester (UK), 2004} (2005), pp. 269--287.

\bibitem{shapirovsky2010simulation}
Shapirovsky, I., \emph{Simulation of two dimensions in unimodal logics}, in:
  P.~Balbiani, N.-Y. Suzuki, F.~Wolter and M.~Zakharyaschev, editors,
  \emph{Advances in Modal Logic, papers from the 8th conference, held in
  Moscow, 2010} (2010), pp. 371--391.

\bibitem{DBLP:conf/aiml/ShapirovskyS02}
Shapirovsky, I. and V.~B. Shehtman, \emph{Chronological future modality in
  {M}inkowski spacetime}, in: F.~Wolter, H.~Wansing, M.~de~Rijke and
  M.~Zakharyaschev, editors, \emph{Advances in Modal Logic, papers from the 4th
  conference, held in Toulouse, 2002} (2003), pp. 437--460.

\bibitem{Shehtman1983}
Shehtman, V.~B., \emph{Modal logics of domains on the real plane}, Studia
  Logica \textbf{42} (1983), pp.~63--80.

\bibitem{Spaan1993}
Spaan, E., \emph{The complexity of propositional tense logics}, in:
  M.~de~Rijke, editor, \emph{Diamonds and Defaults: Studies in Pure and Applied
  Intensional Logic}, Springer, 1993 pp. 287--307.

\bibitem{Szelepcsenyi1988}
Szelepcs\'enyi, R., \emph{The method of forced enumeration for nondeterministic
  automata}, Acta Informatica \textbf{26} (1988), pp.~279--284.

\bibitem{venema1990}
Venema, Y., \emph{Expressiveness and completeness of an interval tense logic},
  Notre Dame Journal of Formal Logic \textbf{31} (1990), pp.~529--547.

\end{thebibliography}

\end{document}